\documentclass[11pt,a4paper,oneside,english,british]{amsart}

\usepackage[T1]{fontenc}
\usepackage[applemac]{inputenc}
\usepackage{babel}
\usepackage{textcomp}
\usepackage{amstext}
\usepackage{amsthm}
\usepackage{amssymb}
\usepackage{stmaryrd}
\usepackage[unicode=true,
 bookmarks=false,
 breaklinks=false,pdfborder={0 0 1},backref=section,colorlinks=false]
 {hyperref}

\makeatletter

\pdfpageheight\paperheight
\pdfpagewidth\paperwidth

\providecommand{\tabularnewline}{\\}

\numberwithin{equation}{section}
\numberwithin{figure}{section}
\theoremstyle{plain}
\newtheorem{thm}{\protect\theoremname}
\theoremstyle{definition}
\newtheorem{defn}[thm]{\protect\definitionname}
\theoremstyle{plain}
\newtheorem{prop}[thm]{\protect\propositionname}
\theoremstyle{remark}
\newtheorem{rem}[thm]{\protect\remarkname}
\theoremstyle{plain}
\newtheorem{lem}[thm]{\protect\lemmaname}

\theoremstyle{plain}

\usepackage{pgf}
\usepackage{tikz}\usetikzlibrary{arrows}

\usepackage{mathrsfs}

\usepackage{babel}
\addto\captionsbritish{\renewcommand{\propositionname}{Proposition}}
\addto\captionsbritish{\renewcommand{\theoremname}{Theorem}}
\addto\captionsenglish{\renewcommand{\propositionname}{Proposition}}
\addto\captionsenglish{\renewcommand{\theoremname}{Theorem}}
\providecommand{\propositionname}{Proposition}
\providecommand{\theoremname}{Theorem}

\AtBeginDocument{
  
}

\makeatother

\addto\captionsbritish{\renewcommand{\definitionname}{Definition}}
\addto\captionsbritish{\renewcommand{\lemmaname}{Lemma}}
\addto\captionsbritish{\renewcommand{\propositionname}{Proposition}}
\addto\captionsbritish{\renewcommand{\remarkname}{Remark}}
\addto\captionsbritish{\renewcommand{\theoremname}{Theorem}}
\addto\captionsenglish{\renewcommand{\definitionname}{Definition}}
\addto\captionsenglish{\renewcommand{\lemmaname}{Lemma}}
\addto\captionsenglish{\renewcommand{\propositionname}{Proposition}}
\addto\captionsenglish{\renewcommand{\remarkname}{Remark}}
\addto\captionsenglish{\renewcommand{\theoremname}{Theorem}}
\providecommand{\definitionname}{Definition}
\providecommand{\lemmaname}{Lemma}
\providecommand{\propositionname}{Proposition}
\providecommand{\remarkname}{Remark}
\providecommand{\theoremname}{Theorem}

\begin{document}
\title{planktons discrete\,-time dynamical systems}
\author{u. a. rozikov, s. k. shoyimardonov, r. varro }
\begin{abstract}
In this paper, we initiate the study of a discrete-time dynamical
system modelling a trophic network connecting the three types
of plankton (phytoplankton, zooplankton, mixoplankton) and bacteria.
The nonlinear operator $V$ associated with this dynamical system
is of type 4-Volterra quadratic stochastic operator (QSO) with
twelve parameters.  We give conditions on the parameters under which
this operator maps the five-dimensional
standard simplex to itself and we find
its fixed points. Moreover, we study the limit points of trajectories for this
operator. For each situations we give
some biological interpretations.
\end{abstract}

\subjclass[2000]{34D20 (92D25).}
\keywords{Quadratic stochastic operator, Volterra operator, $\ell$-Volterra
operator, plankton, mixoplankton, ocean ecosystem.}
\address{Utkir Rozikov. \ \  V.I.Romanovskiy institute of mathematics, 81, Mirzo Ulug'bek str.,
100125, Tashkent, Uzbekistan.}
\email{rozikovu@yandex.ru}
\address{Sobirjon Shoyimardonov. \ \ V.I.Romanovskiy institute of mathematics, 81, Mirzo Ulug'bek str.,
100125, Tashkent, Uzbekistan.}
\email{shoyimardonov@inbox.ru}
\address{Richard Varro. \foreignlanguage{english}{Institut Montpelliérain
Alexander Grothendieck, Université de Montpellier, CNRS, Place
Eugène Bataillon - 35095 Montpellier, France.}}
\email{\selectlanguage{english}%
{\footnotesize{}richard.varro@umontpellier.fr}}

\maketitle
\selectlanguage{british}%

\section{Introduction}

The conventional view of species divides organisms into two distinct
classes:
\begin{itemize}
\item {\it Autotroph} species like
plants, are able to synthesize all the organic matter that composes
them from mineral elements taken from the environment.
\item {\it Heterotroph}
organisms like animals, must absorb preformed organic matter
present in their environment to constitute the organic matter
that composes them.

Beside these two modes of trophism there
is a third:
\item {\it The mixotrophy}, that combines the other two modes.
Mixothroph organisms are able to constitute their organic matter
by autotrophy as well as by heterotrophy. Mixothrophy is very
rare among terrestrial organisms, but since the 1980s many studies
have shown that mixothrophy is very abundant in ocean plankton
\cite{Stoecker-17}.
\end{itemize}
Plankton plays important role in marine ecosystems, it provides the
base for the aquatic food chain \cite{Flynn-Mitra-1,Leles,Stickney}.

Generally\footnote{see https://en.wikipedia.org/wiki/Phytoplankton \ \ and references therein.}
small or microscopic, the plankton organisms are capable
of limited movement, but unable to move against the current and
drift with the currents.

There are two types of plankton: plant-like
plankton also called {\it phytoplankton} and animal-like plankton called
{\it zooplankton}.

Phytoplankton mainly consist of unicellular photosynthetic
organisms, therefore autotroph organisms absorbing mineral elements
such as nitrogen, phosphorus, calcium, iron and transform these
elements into organic matter using energy from sunlight. It is
estimated that phytoplankton contributes about half of the photosynthesis
on the planet and absorbs one-third of the carbon dioxide.

Zooplankton includes holoplankton or permanent zooplankton composed
of organisms that are born, reproduce and die as zooplankton
and meroplankton or temporary zooplankton consisting of eggs
and larvae of many species (fish, crustaceans, shells ...) that
leave the plankton stage when they metamorphose. Zooplankton
feed on organic matter, phytoplankton or zooplankton.

Research over the last decade showed that zooplankton is subdivided into
protozooplankton (incapable of phototrophy) and mixoplankton
(capable of phototrophy and phagotrophy). Long regarded as a
curiosity, it appears today that mixotrophy is common \cite{Flynn-Mitra-2}
and that many organisms considered as phytoplankton are mixotroph
and that 50\% of zooplankton is mixoplankton. Mixoplankton is
subdivided into three functional groups: constitutive mixoplankton
whose physiology allows photosynthesis, general non-constitutive
mixoplankton which extracts the organelles necessary for photosynthesis
in their preys and specialist non-constitutive mixoplankton which
attack specific prey to extract certain parts allowing photosynthesis.

When they die, phytoplankton, zooplankton and mixoplankton elements
are decomposed into dissolved organic matter (DOM) that are consumed
and transformed by bacteria into dissolved inorganic matter (DIM)
that are in turn consumed by phytoplankton, mixoplankton and
bacteria.

This cycle \cite{Mitra-16} is summarized in the following diagram.

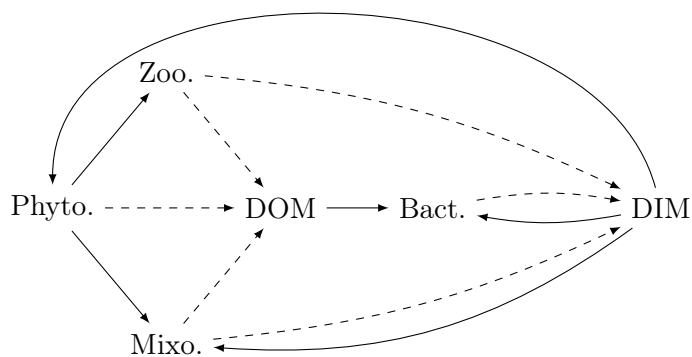
\begin{center}
\begin{tikzpicture}[scale=1.0]
\node (P) at (-3,0) {Phyto.};
\node (M) at (-1.5,-1.8) {Mixo.};
\node (Z) at (-1.5,1.8) {Zoo.};
\node (DOM) at (0,0) {DOM};
\node (B) at (2,0) {Bact.};
\node (DIM) at (5,0) {DIM};
\draw[->,>=latex] (P) -- (M);
\draw[->,>=latex] (P) -- (Z);
\draw [dashed][->,>=latex] (P) -- (DOM);
\draw [dashed][->,>=latex] (M) -- (DOM);
\draw [dashed][->,>=latex] (Z) -- (DOM);
\draw [dashed][->,>=latex] (M) to[bend right=10] (DIM);
\draw[->,>=latex] (DIM) to[bend left=20] (M);
\draw[->,>=latex] (DOM) -- (B);
\draw [dashed][->,>=latex] (Z) to[bend left=10] (DIM);
\draw [dashed][->,>=latex] (B) to[bend left=10] (DIM);
\draw[->,>=latex] (DIM) to[bend left=10] (B);
\draw[->,>=latex] (DIM) to[out=105,in=90] (P);
\end{tikzpicture}
\par\end{center}

If the quantities, measured by concentrations or weights, of
phytoplankton, zooplankton, mixoplankton, bacteria, dissolved
organic matter and dissolved inorganic matter are respectively
denoted $P$, $Z$, $M$, $B$, $O$ and $I$. And if we denote
$a_{1}$, \dots, $a_{12}$ the transfer rates between the different
compartments according to the graph below.
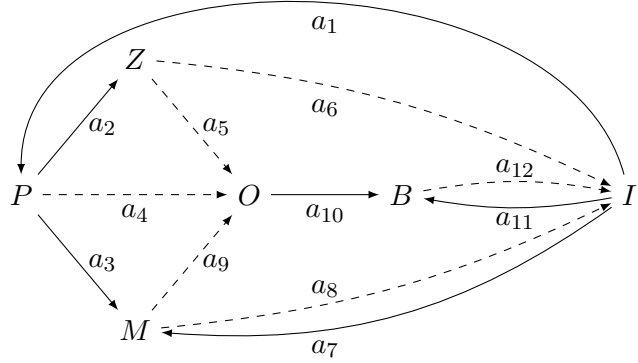
\begin{center}
\begin{tikzpicture}[scale=1.0]
\node (P) at (-3,0) {$P$};
\node (Z) at (-1.5,1.8) {$Z$};
\node (M) at (-1.5,-1.8) {$M$};
\node (COD) at (0,0) {$O$};
\node (B) at (2,0) {$B$};
\node (COI) at (5,0) {$I$};
\draw[->,>=latex] (P) -- (M) node[midway,right] {$a_{3}$};
\draw[->,>=latex] (P) -- (Z) node[midway,right] {$a_{2}$};
\draw [dashed][->,>=latex] (P) -- (COD) node[midway,below] {$a_{4}$};
\draw [dashed][->,>=latex] (M) -- (COD) 	node[midway,right] {$a_{9}$};
\draw [dashed][->,>=latex] (Z) -- (COD) 	node[midway,right] {$a_{5}$};
\draw[->,>=latex] (COI) to[bend left=20] (M);
    \draw (1,-1.8) node[below]{$a_{7}$};
\draw [dashed][->,>=latex] (M) to[bend right=10] (COI);
	\draw (1,-1) node[below]{$a_{8}$};
\draw[->,>=latex] (COD) -- (B) node[midway,below] {$a_{10}$};
\draw [dashed][->,>=latex] (Z) to[bend left=10] (COI) ;
    \draw (1,1.4) node[below]{$a_{6}$};
\draw [dashed][->,>=latex] (B) to[bend left=10] (COI);
	\draw (3.5,0.6) node[below]{$a_{12}$};
\draw[->,>=latex] (COI) to[bend left=10] (B);
    \draw (3.5,-0.1) node[below]{$a_{11}$};
\draw[->,>=latex] (COI) to[out=105,in=90] (P) ;
    \draw (1,2.5) node[below]{$a_{1}$};
\end{tikzpicture}
\par\end{center}

Then the corresponding model is
\begin{equation}
\begin{cases}
\frac{dP}{dt} & =a_{1}PI-a_{2}PZ-a_{3}PM-a_{4}P\\[2mm]
\frac{dZ}{dt} & =a_{2}PZ-a_{5}Z-a_{6}Z\\[2mm]
\frac{dM}{dt} & =a_{3}PM+a_{7}MI-a_{8}M-a_{9}M\\[2mm]
\frac{dO}{dt} & =a_{4}P+a_{5}Z+a_{9}M-a_{10}BO\\[2mm]
\frac{dB}{dt} & =a_{10}BO+a_{11}BI-a_{12}B\\[2mm]
\frac{dI}{dt} & =a_{6}Z+a_{8}M+a_{12}B-a_{1}PI-a_{7}MI-a_{11}BI
\end{cases}\label{eq:Eq1}
\end{equation}

\begin{rem} We note that the study of the role of mixoplankton is very
recent and the system (\ref{eq:Eq1}) is not analytically studied yet.
To the best of our knowledge this paper is a first exploration of the field.
It remains to be studied what happens when consumption rates are higher than death rates.
\end{rem}

For simplicity we have assumed that the dissolved inorganic matter
produced by bacteria is immediately assimilated and available
for food.

We notice that $\frac{d}{dt}\left(P+Z+M+O+B+I\right)=0$, from
this we deduce that the total weight of this closed system remains
constant over time, by convention we put $P+Z+M+O+B+I=1$. Coefficients
$a_{4}$, $a_{5}$, $a_{9}$ are decomposition rates in dissolved
organic matter and $a_{6}$, $a_{8}$, $a_{12}$ decomposition
rates in dissolved inorganic matter, we can assume that these
coefficients are strictly positive and interpret $a_{4}$, $a_{5}+a_{6}$,
$a_{8}+a_{9}$, $a_{12}$ as death rates. Similarly, the coefficients
$a_{1}$, $a_{2}$, $a_{3}$, $a_{7}$, $a_{10}$ and $a_{11}$
being consumption rates, we can also assume that these coefficients
are also strictly positive.

\section{$\ell$-Volterra Quadratic Stochastic Operators}

\emph{The quadratic stochastic operator} (QSO) \cite{GMR}, \cite{L} is
a mapping of the standard simplex.
\begin{equation}
S^{m-1}=\{x=(x_{1},...,x_{m})\in\mathbb{R}^{m}:x_{i}\geq0,\sum\limits _{i=1}^{m}x_{i}=1\}\label{2}
\end{equation}
into itself, of the form
\begin{equation}
V:x'_{k}=\sum\limits _{i=1}^{m}\sum\limits _{j=1}^{m}P_{ij,k}x_{i}x_{j},\qquad k=1,...,m,\label{3}
\end{equation}
where the coefficients $P_{ij,k}$ satisfy the following conditions
\begin{equation}
P_{ij,k}\geq0,\quad P_{ij,k}=P_{ji,k},\quad\sum\limits _{k=1}^{m}P_{ij,k}=1,\qquad(i,j,k=1,...,m).\label{4}
\end{equation}

Thus, each quadratic stochastic operator $V$ can be uniquely
defined by a cubic matrix $\mathbb{P}=(P_{ij,k})_{i,j,k=1}^{m}$
with conditions (\ref{4}).

Note that each element $x\in S^{m-1}$ is a probability distribution
on $\left\llbracket 1,m\right\rrbracket =\{1,...,m\}.$
Each such distribution can be interpreted as a state of the corresponding
biological system.

For a given $\lambda^{(0)}\in S^{m-1}$ the \emph{trajectory}
(orbit) $\{\lambda^{(n)};n\geq0\}$ of $\lambda^{(0)}$ under
the action of QSO (\ref{3}) is defined by
\[
\lambda^{(n+1)}=V(\lambda^{(n)}),\;n=0,1,2,...
\]

The main problem in mathematical biology consists in the
study of the asymptotical behaviour of the trajectories. The
difficulty of the problem depends on given matrix $\mathbb{P}$.
One of simple cases is Volterra QSO, i.e., the matrix $\mathbb{P}$
satisfies
\begin{equation}
P_{ij,k}=0\text{ if }k\notin\{i,j\}\mbox{ for any }i,j\in\left\llbracket 1,m\right\rrbracket .\label{v}
\end{equation}
In \cite{RGan} the theory for Volterra QSO was developed by
using the theories of the Lyapunov function and of tournaments.
But non-Volterra QSO (i.e., not satisfying condition (\ref{v}))
were not exhaustively studied, because there is no general theory
that can be applied to the study of non-Volterra operators (see \cite{GMR} for a review).
In this paper we reduce our operator to a non-Volterra operator which is known as a $\ell$-Volterra QSO (see \cite{RZ}, \cite{Roz1}, \cite{Roz4}):

\emph{$\ell$-Volterra QSO}.
Fix $\ell\in\left\llbracket 1,m\right\rrbracket $ and assume
that elements $P_{ij,k}$ of the matrix $\mathbb{P}$ satisfy
\begin{equation}
\begin{array}{l}
P_{ij,k}=0{\rm \text{ if }}k\notin\{i,j\}\mbox{ for any }k\in\{1,...,\ell\},i,j\in\left\llbracket 1,m\right\rrbracket ;\\[3mm]
P_{ij,k}>0\mbox{ for at least one pair }(i,j),i\neq k,j\neq k\\
\hspace{14mm}\mbox{ for any }k\in\{\ell+1,...,m\}.
\end{array}\label{7}
\end{equation}

\begin{defn}
\label{d1} \cite{RZ} For any fixed $\ell\in\left\llbracket 1,m\right\rrbracket $,
the QSO defined by (\ref{3}), (\ref{4}) and (\ref{7}) is called
$\ell$-Volterra QSO.
\end{defn}

\begin{defn}
A QSO $V$ is called regular if for any initial point $\lambda^{(0)}\in S^{m-1}$,
the limit
\[
\lim_{n\to\infty}V^{n}(\lambda^{(0)})
\]
exists, where $V^n$ denotes $n$-fold composition of $V$ with itself (i.e. $n$ time iterations of $V$).
\end{defn}

\section{Reduction to 2-Volterra QSO}

In this paper we study the discrete time dynamical systems associated to the system
(\ref{eq:Eq1}).

For simplicity of notations we denote
$$x_{1}=P, \ \ x_{2}=Z, \ \ x_{3}=M, \ \ x_{4}=B, \ \ x_{5}=O, \ \ x_{6}=I.$$
Define the evolution operator $V$ by $
V:S^{5}\rightarrow\mathbb{R}^{6},\quad\left(x_{1},\ldots,x_{6}\right)\mapsto\left(x'_{1},\ldots,x'_{6}\right)
$
\begin{equation}
V:\left\{ \begin{alignedat}{1}x'_{1} & =x_{1}\left(1-a_{4}+a_{1}x_{6}-a_{2}x_{2}-a_{3}x_{3}\right)\\
x'_{2} & =x_{2}\left(1-a_{5}-a_{6}+a_{2}x_{1}\right)\\
x'_{3} & =x_{3}\left(1-a_{8}-a_{9}+a_{3}x_{1}+a_{7}x_{6}\right)\\
x'_{4} & =x_{4}\left(1-a_{12}+a_{10}x_{5}+a_{11}x_{6}\right)\\
x'_{5} & =x_{5}+a_{4}x_{1}+a_{5}x_{2}+a_{9}x_{3}-a_{10}x_{4}x_{5}\\
x'_{6} & =x_{6}\left(1-a_{1}x_{1}-a_{7}x_{3}-a_{11}x_{4}\right)+a_{6}x_{2}+a_{8}x_{3}+a_{12}x_{4}
\end{alignedat}
\right.\label{disc}
\end{equation}

Note that the operator $V$ has a form of 4-Volterra QSO, but the
parameters of this operator are not related to $P_{ij,k}$. Here
to make some relations with $P_{ij,k}$ we find conditions on
parameters of (\ref{disc}) rewriting it in the form (\ref{3})
(as in \cite{RSH}). Using $x_{1}+\cdots+x_{6}=1$ we change
the form of the operator (\ref{disc}) as following:
\[
V:\left\{ \begin{alignedat}{1}x'_{1} & =x_{1}\left((1-a_{4})(x_{1}+\cdots+x_{6})+a_{1}x_{6}-a_{2}x_{2}-a_{3}x_{3}\right)\\
x'_{2} & =x_{2}\left((1-a_{5}-a_{6})(x_{1}+\cdots+x_{6})+a_{2}x_{1}\right)\\
x'_{3} & =x_{3}\left((1-a_{8}-a_{9})(x_{1}+\cdots+x_{6})+a_{3}x_{1}+a_{7}x_{6}\right)\\
x'_{4} & =x_{4}\left((1-a_{12})(x_{1}+\cdots+x_{6})+a_{10}x_{5}+a_{11}x_{6}\right)\\
x'_{5} & =(x_{5}+a_{4}x_{1}+a_{5}x_{2}+a_{9}x_{3})(x_{1}+\cdots+x_{6})-a_{10}x_{4}x_{5}\\
x'_{6} & =x_{6}\left(x_{1}+\cdots+x_{6}-a_{1}x_{1}-a_{7}x_{3}-a_{11}x_{4}\right)+\\
 & \quad(a_{6}x_{2}+a_{8}x_{3}+a_{12}x_{4})(x_{1}+\cdots+x_{6})
\end{alignedat}
\right.
\]

From this system and QSO (\ref{3}) for the case $m=6$ we obtain the
following relations between $P_{ij,k}$:

\begin{equation}
\begin{array}{cccc}
\begin{aligned}{\scriptstyle P_{11,1}} & ={\scriptstyle 1-a_{4}},\\
{\scriptstyle 2P_{14,1}} & ={\scriptstyle 1-a_{4}},\\
{\scriptstyle 2P_{12,2}} & ={\scriptstyle 1+a_{2}-a_{5}-a_{6}},\\
{\scriptstyle 2P_{24,2}} & ={\scriptstyle 1-a_{5}-a_{6}},\\
{\scriptstyle 2P_{13,3}} & ={\scriptstyle 1+a_{3}-a_{8}-a_{9}},\\
{\scriptstyle 2P_{34,3}} & ={\scriptstyle 1-a_{8}-a_{9}},\\
{\scriptstyle 2P_{14,4}} & ={\scriptstyle 1-a_{12}},\\
{\scriptstyle P_{44,4}} & ={\scriptstyle 1-a_{12}},\\
{\scriptstyle P_{11,5}} & ={\scriptstyle a_{4}},\\
{\scriptstyle 2P_{14,5}} & ={\scriptstyle a_{4}},\\
{\scriptstyle P_{22,5}} & ={\scriptstyle a_{5}},\\
{\scriptstyle 2P_{25,5}} & ={\scriptstyle 1+a_{5}},\\
{\scriptstyle 2P_{34,5}} & ={\scriptstyle a_{9}},\\
{\scriptstyle 2P_{45,5}} & ={\scriptstyle 1-a_{10}},\\
{\scriptstyle 2P_{12,6}} & ={\scriptstyle a_{6}},\\
{\scriptstyle 2P_{16,6}} & ={\scriptstyle 1-a_{1}},\\
{\scriptstyle 2P_{24,6}} & ={\scriptstyle a_{6}+a_{12}},\\
{\scriptstyle P_{33,6}} & ={\scriptstyle a_{8}},\\
{\scriptstyle 2P_{36,6}} & ={\scriptstyle 1-a_{7}+a_{8}},\\
{\scriptstyle 2P_{46,6}} & ={\scriptstyle 1-a_{11}+a_{12}},\\

\end{aligned}

\begin{aligned}{\scriptstyle 2P_{12,1}} & ={\scriptstyle 1-a_{2}-a_{4}},\\
 {\scriptstyle 2P_{15,1}} & ={\scriptstyle 1-a_{4}},\\
 {\scriptstyle P_{22,2}} & ={\scriptstyle 1-a_{5}-a_{6}},\\
 {\scriptstyle 2P_{25,2}} & ={\scriptstyle 1-a_{5}-a_{6}},\\
{\scriptstyle 2P_{23,3}} & ={\scriptstyle 1-a_{8}-a_{9}},\\
{\scriptstyle 2P_{35,3}} & ={\scriptstyle 1-a_{8}-a_{9}},\\
{\scriptstyle 2P_{24,4}} & ={\scriptstyle 1-a_{12}},\\
{\scriptstyle 2P_{45,4}} & ={\scriptstyle 1+a_{10}-a_{12}},\\
{\scriptstyle 2P_{12,5}} & ={\scriptstyle a_{4}+a_{5}},\\
{\scriptstyle 2P_{15,5}} & ={\scriptstyle 1+a_{4}},\\
{\scriptstyle 2P_{23,5}} & ={\scriptstyle a_{5}+a_{9}},\\
{\scriptstyle 2P_{26,5}} & ={\scriptstyle a_{5}},\\
{\scriptstyle 2P_{35,5}} & ={\scriptstyle 1+a_{9}},\\
{\scriptstyle P_{55,5}} & ={\scriptstyle 1},\\
{\scriptstyle 2P_{13,6}} & ={\scriptstyle a_{8}},\\
{\scriptstyle P_{22,6}} & ={\scriptstyle a_{6}},\\
{\scriptstyle 2P_{25,6}} & ={\scriptstyle a_{6}},\\
{\scriptstyle 2P_{34,6}} & ={\scriptstyle a_{8}+a_{12}},\\
{\scriptstyle P_{44,6}} & ={\scriptstyle a_{12}},\\
{\scriptstyle 2P_{56,6}} & ={\scriptstyle 1},\\

\end{aligned}

 & \begin{aligned}{\scriptstyle 2P_{13,1}} & ={\scriptstyle 1-a_{3}-a_{4}},\\
 {\scriptstyle 2P_{16,1}} & ={\scriptstyle 1+a_{1}-a_{4}},\\
 {\scriptstyle 2P_{23,2}} & ={\scriptstyle 1-a_{5}-a_{6}},\\
{\scriptstyle 2P_{26,2}} & ={\scriptstyle 1-a_{5}-a_{6}},\\
{\scriptstyle P_{33,3}} & ={\scriptstyle 1-a_{8}-a_{9}},\\
{\scriptstyle 2P_{36,3}} & ={\scriptstyle 1+a_{7}-a_{8}-a_{9}},\\
{\scriptstyle 2P_{34,4}} & ={\scriptstyle 1-a_{12}},\\
{\scriptstyle 2P_{46,4}} & ={\scriptstyle 1+a_{11}-a_{12}},\\
{\scriptstyle 2P_{13,5}} & ={\scriptstyle a_{4}+a_{9}},\\
{\scriptstyle 2P_{16,5}} & ={\scriptstyle a_{4}},\\
{\scriptstyle 2P_{24,5}} & ={\scriptstyle a_{5}},\\
{\scriptstyle P_{33,5}} & ={\scriptstyle a_{9}},\\
{\scriptstyle 2P_{36,5}} & ={\scriptstyle a_{9}},\\
{\scriptstyle 2P_{56,5}} & ={\scriptstyle 1},\\
{\scriptstyle 2P_{14,6}} & ={\scriptstyle a_{12}},\\
{\scriptstyle 2P_{23,6}} & ={\scriptstyle a_{6}+a_{8}},\\
{\scriptstyle 2P_{26,6}} & ={\scriptstyle 1+a_{6}},\\
{\scriptstyle 2P_{35,6}} & ={\scriptstyle a_{8}},\\
{\scriptstyle 2P_{45,6}} & ={\scriptstyle a_{12}},\\
{\scriptstyle P_{66,6}} & ={\scriptstyle 1},\\

\end{aligned}
\end{array}\label{par}
\end{equation}

${\scriptstyle \text{other }} {\scriptstyle P_{ij,k}=0}.$

\begin{prop}\label{pc}
We have $V\left(S^{5}\right)\subset S^{5}$ if the
strictly positive parameters $a_{1}, \ldots, a_{12}$
verify the following conditions
\begin{equation}
\begin{array}{cccc}
a_{1}\leq1, & a_{6}\leq1, & a_{10}\leq1, & a_{12}\leq1,\medskip\\
a_{2}+a_{4}\leq1, & a_{3}+a_{4}\leq1, & a_{5}+a_{6}\leq1, & a_{8}+a_{9}\leq1,\medskip\\
 & \left|a_{7}-a_{8}\right|\leq1, & \left|a_{11}-a_{12}\right|\leq1.
\end{array}\label{cond}
\end{equation}

Moreover, under conditions (\ref{cond}) the operator $V$ is
a 4-Volterra QSO.
\end{prop}
\begin{proof} The proof can be obtained  by using equalities (\ref{par}) and solving
inequalities $0\leq P_{ij,k}\leq1$ for each $P_{ij,k}$.
\end{proof}

\begin{rem}
In the sequel of the paper we consider operator (\ref{disc})
with parameters $a_{1},...,a_{12}$ which satisfy conditions
(\ref{cond}). This operator maps $S^{5}$ to itself and we are
interested to study the behaviour of the trajectory of any initial
point $x\in S^{5}$ under iterations of the operator $V.$
\end{rem}
\begin{rem} The conditions of Proposition \ref{pc} are sufficient for the quadratic operator to
map simplex to itself. A criterion for a quadratic operator with real coefficients $\gamma_{ij,k}$ (not necessary positive)
to map the $(n-1)$-dimensional simplex  to itself is known (\cite{SO}):
\begin{itemize}
\item[i)] $\sum_{k=1}^n\gamma_{ij,k}=1$, for all $i,j=1, \dots, n$;
\item[ii)] $0\leq \gamma_{ii,k}\leq 1$, for all $i,k=1, \dots, n$;
\item[iii)] $-\sqrt{\gamma_{ii,k}\gamma_{jj,k}}\leq \gamma_{ij,k}\leq 1+\sqrt{(1-\gamma_{ii,k})(1-\gamma_{jj,k})}.$
\end{itemize}
\end{rem}
\section{Fixed points of the operator (\ref{disc})}
\begin{defn}
\cite{De}. A fixed point $p$ for $F:\mathbb{R}^{m}\rightarrow\mathbb{R}^{m}$
is called \emph{hyperbolic} if the Jacobian matrix $\textbf{J}=\textbf{J}_{F}$
of the map $F$ at the point $p$ has no eigenvalues on the unit
circle.

There are three types of hyperbolic fixed points:

(1) $p$ is an attracting fixed point if all of the eigenvalues
of $\textbf{J}(p)$ are less than one in absolute value.

(2) $p$ is an repelling fixed point if all of the eigenvalues
of $\textbf{J}(p)$ are greater than one in absolute value.

(3) $p$ is a saddle point otherwise.
\end{defn}

To find fixed points of operator $V$ given by (\ref{disc})
we have to solve $V(x)=x.$

By the following proposition we give all possible fixed points
of the operator $V.$
\begin{prop}\label{fixp}
\label{fp} The operator (\ref{disc}) has the following fixed
points with conditions (\ref{cond}) for parameters:
\begin{itemize}
\item[1)] $\lambda_{1}=\left(0,0,0,0,\lambda,1-\lambda\right)$ with
$\lambda\in[0,1]$.

\item[2)] $\lambda_{2}=\left(0,0,0,1-\frac{a_{12}}{a_{11}},0,\frac{a_{12}}{a_{11}}\right)$
if $a_{12}\leq a_{11}$.

\item[3)] $\lambda_{3}=\left(0,0,\frac{a_{7}a_{12}-a_{9}a_{11}-a_{8}a_{11}}{a_{7}a_{9}}\lambda,\lambda,\frac{a_{7}a_{12}
    -a_{9}a_{11}-a_{8}a_{11}}{a_{7}a_{10}},\frac{a_{9}+a_{8}}{a_{7}}\right)$
if
\[
a_{12}a_{7}-a_{11}(a_{8}+a_{9})\geq0\text{ and }a_{7}\left(a_{10}-a_{12}\right)-\left(a_{8}+a_{9}\right)\left(a_{10}-a_{11}\right)\geq0
\]
where
\[
\lambda=\tfrac{a_{9}\left(a_{7}a_{10}-a_{9}a_{10}-a_{8}a_{10}-a_{7}a_{12}+a_{9}a_{11}+
a_{8}a_{11}\right)}{a_{10}\left(a_{7}a_{9}+a_{7}a_{12}-a_{8}a_{11}-a_{9}a_{11}\right)}.
\]

\item[4)] $\lambda_{4}=\left(\frac{a_{1}a_{12}-a_{4}a_{11}}{a_{1}a_{4}}\lambda,0,0,\lambda,
    \frac{a_{1}a_{12}-a_{4}a_{11}}{a_{1}a_{10}},\frac{a_{4}}{a_{1}}\right)$
if
\[
a_{12}a_{1}-a_{11}a_{4}\geq0\text{ and }a_{10}a_{1}-a_{10}a_{4}-a_{12}a_{1}+a_{11}a_{4}\geq0,
\]
where
\[
\lambda=\tfrac{a_{4}\left(a_{10}a_{1}-a_{10}a_{4}-a_{12}a_{1}+a_{11}a_{4}\right)}{a_{10}\left(a_{12}a_{1}-a_{11}a_{4}+a_{4}a_{1}\right)}.
\]
\item[5)] $\lambda_{5}=\left(\frac{a_{8}+a_{9}-a_{7}x_{6}}{a_{3}},0,\frac{a_{1}\lambda-
    a_{4}}{a_{3}},\frac{a_{4}a_{8}+\left(a_{1}a_{9}-a_{4}a_{7}\right)\lambda}{a_{3}\left(a_{12}-
    a_{11}\lambda\right)},\frac{a_{12}-a_{11}\lambda}{a_{10}},\lambda\right)$
where $\lambda$ is a root of
\begin{equation}\label{qe}
ax^{2}+bx+c=0
\end{equation} where
\begin{align*}
a & =a_{3}a_{11}^{2}+a_{10}a_{11}(a_{7}-a_{1}-a_{3}),\\
b & =a_{10}a_{11}(a_{3}+a_{4}-a_{8}-a_{9})+a_{10}a_{12}(a_{3}+a_{1}-a_{7})+\\
 & \quad a_{10}(a_{1}a_{9}-a_{4}a_{7})-2a_{3}a_{11}a_{12},\\
c & =a_{10}a_{12}(a_{8}+a_{9}-a_{3}-a_{4})+a_{3}a_{12}^{2}+a_{4}a_{8}a_{10},
\end{align*}
if $a_{4}<a_{1}$ and $\lambda$ verify

\begin{tabular}{cl}
$\tfrac{a_{4}}{a_{1}}<\lambda<\min\{\tfrac{a_{12}}{a_{11}},\tfrac{a_{8}+a_{9}}{a_{7}}\}$ & if $a_{1}a_{9}\geq a_{4}a_{7}$,\medskip\tabularnewline
$\frac{a_{4}}{a_{1}}<\lambda<\min\{\frac{a_{12}}{a_{11}},\frac{a_{8}+a_{9}}{a_{7}},\frac{a_{4}a_{8}}{a_{4}a_{7}-a_{1}a_{9}}\}$ & if $a_{1}a_{9}<a_{4}a_{7}$.\tabularnewline
\end{tabular}
\item[6)] $\lambda_{6}=\left(\frac{a_{5}+a_{6}}{a_{2}},\frac{a_{1}\lambda-a_{4}}{a_{2}},0,\frac{a_{4}a_{6}
    +a_{1}a_{5}\lambda}{a_{2}(a_{12}-a_{11}\lambda)},\frac{a_{12}-a_{11}\lambda}{a_{10}},\lambda\right)$
if $a_{4}<a_{1}$ and $\frac{a_{4}}{a_{1}}<\lambda<\frac{a_{12}}{a_{11}}$,
where $\lambda$ is a root of
\begin{equation}\label{qea}
\alpha x^{2}+\beta x+\gamma=0
\end{equation}
 where
\[
\begin{alignedat}{1}\alpha & =a_{2}a_{11}^{2}-a_{10}a_{11}\left(a_{1}+a_{2}\right),\\
\beta & =a_{10}a_{11}\left(a_{2}+a_{4}-a_{5}-a_{6}\right)+a_{10}a_{12}(a_{1}+a_{2})+a_{1}a_{5}a_{10}-2a_{2}a_{11}a_{12},\\
\gamma & =a_{10}a_{12}\left(a_{5}+a_{6}-a_{2}-a_{4}\right)+a_{2}a_{12}^{2}+a_{4}a_{6}a_{10}.
\end{alignedat}
\]
\item[7)] $\lambda_{7}=\left(\frac{a_{5}+a_{6}}{a_{2}},x_{2},\lambda,x_{4},x_{5},\frac{a_{2}\left(a_{8}
    +a_{9}\right)-a_{3}\left(a_{5}+a_{6}\right)}{a_{2}a_{7}}\right)$
where
\[
\begin{alignedat}{1}x_{2} & =\tfrac{a_{1}\left(a_{2}a_{8}+a_{2}a_{9}-a_{3}a_{5}-a_{3}a_{6}\right)
-a_{2}a_{4}a_{7}-a_{2}a_{3}a_{7}\lambda}{a_{2}^{2}a_{7}},\\
x_{4} & =\tfrac{a_{7}\left(a_{4}a_{5}+a_{4}a_{6}+a_{2}a_{5}x_{2}+a_{2}a_{9}\lambda\right)}
{a_{2}a_{7}a_{12}-a_{11}\left(a_{2}a_{8}+a_{2}a_{9}-a_{3}a_{5}-a_{3}a_{6}\right)},\\
x_{5} & =\tfrac{a_{2}a_{7}a_{12}-a_{11}\left(a_{2}a_{8}+a_{2}a_{9}-a_{3}a_{5}-a_{3}a_{6}\right)}{a_{2}a_{7}a_{10}},\\
\lambda& =1-\tfrac{a_{5}+a_{6}}{a_{2}}-x_{2}-x_{4}-x_{5}-\tfrac{a_{2}\left(a_{8}+a_{9}\right)
-a_{3}\left(a_{5}+a_{6}\right)}{a_{2}a_{7}},
\end{alignedat}
\]
if
\[
\tfrac{a_{2}a_{4}a_{7}}{a_{1}}<a_{2}a_{8}+a_{2}a_{9}-a_{3}a_{5}-a_{3}a_{6}<\tfrac{a_{2}a_{7}a_{12}}{a_{11}}
\]
and
\[
\lambda<\tfrac{a_{1}\left(a_{2}a_{8}+a_{2}a_{9}-a_{3}a_{5}-a_{3}a_{6}\right)-a_{2}a_{4}a_{7}}{a_{2}a_{3}a_{7}}.
\]
\end{itemize}
\end{prop}
\begin{proof}
Recall that a fixed point of the operator $V$ is a solution
to $V(x)=x.$

If $x_{6}=0$, from $x_{6}^{\left(1\right)}=x_{6}$ in (\ref{disc})
and $a_{6},a_{8},a_{9}>0$ we deduce $x_{2}=x_{3}=x_{4}=0$,
therefore with $x'_{1}=x_{1}$ and $a_{4}>0$ we get $x_{1}=0$
and from $x_{1}+\cdots+x_{6}=0$ we deduce $x_{5}=1$ and the
fixed point $\left(0,0,0,0,1,0\right)$.

We assume now that $x_{6}\neq0$.

1) If $x_{4}=0$, from $x_{5}^{\left(1\right)}=x_{5}$ in (\ref{disc})
we get $x_{1}=x_{2}=x_{3}=0$ and from $x_{5}+x_{5}=1$ we find
fixed points $\lambda_{1}$.\medskip{}

2) If $x_{4}\neq0$ and $x_{5}=0$, from $x_{5}^{\left(1\right)}=x_{5}$
in (\ref{disc}) we get $x_{1}=x_{2}=x_{3}=0$, with this and
$x_{6}^{\left(1\right)}=x_{6}$ we have $x_{6}=\frac{a_{12}}{a_{11}}$
from where $x_{4}=1-x_{6}$, we find the fixed point $\lambda_{2}$.

We assume now $x_{4}\neq0$ and $x_{5}\neq0$.

3) If $x_{1}=0$, from $x_{2}^{\left(1\right)}=x_{2}$ we get
$x_{2}=0$, then from $x_{5}^{\left(1\right)}=x_{5}$ we get
$x_{3}\neq0$, it follows from $x_{3}^{\left(1\right)}=x_{3}$
that $x_{6}=\frac{a_{8}+a_{9}}{a_{7}}$, next with $x_{4}^{(1)}=x_{4}$
we get $x_{5}=\frac{a_{12}-a_{11}x_{6}}{a_{10}}=\frac{(a_{12}a_{7}-a_{11}a_{9}-a_{11}a_{8})}{a_{7}a_{10}}$
and with $x_{5}^{(1)}=x_{5}$ we get $x_{3}=\frac{(a_{12}-a_{11}x_{6})}{a_{9}}x_{4}=\frac{a_{12}a_{7}-a_{11}a_{9}-a_{11}a_{8}}{a_{7}a_{9}}x_{4}$,
and by $x_{1}+\cdots+x_{6}=1$ we obtain $x_{4}=\frac{a_{9}(a_{7}a_{10}-a_{9}a_{10}-a_{8}a_{10}-a_{12}a_{7}+a_{11}a_{9}+a_{11}a_{8})}{a_{10}(a_{12}a_{7}-a_{11}a_{9}-a_{11}a_{8}+a_{9}a_{7})}$,
thus we have the fixed point $\lambda_{3}.$ Here provided that
all parameters are positive and since the sum of them equal to
one, we can say that they are less than one.

4) If $x_{1}\neq0$. In the case $x_{2}=0$ and $x_{3}=0$, let
$x_{4}=\lambda\neq0$, then from $x_{1}^{(1)}=x_{1}$ we get
$x_{6}=\frac{a_{4}}{a_{1}}$, from $x_{4}^{(1)}=x_{4}$ we get
$x_{5}=\frac{a_{12}a_{1}-a_{11}a_{4}}{a_{10}a_{1}}$ and from
$x_{6}^{(1)}=x_{6}$ we deduce $x_{1}=\frac{(a_{12}a_{1}-a_{1}a_{4})}{a_{4}a_{1}}\lambda$.
In all cases $x_{1}+\cdots+x_{6}=1$ implies $\frac{(a_{12}a_{1}-a_{1}a_{4})}{a_{1}a_{4}}\lambda+\lambda+\frac{(a_{12}a_{1}-a_{11}a_{4})}{a_{10}a_{1}}+\frac{a_{4}}{a_{1}}=1$,
from this we find $\lambda=\frac{a_{4}(a_{10}a_{1}-a_{10}a_{4}-a_{12}a_{1}+a_{11}a_{4})}{a_{10}(a_{12}a_{1}-a_{11}a_{4}+a_{4}a_{1})}$
and thus we have $\lambda_{4}.$ By conditions to the parameters,
all coordinates are positive. Moreover, sum of them equal to
one. It means that the fixed point belongs to the simplex $S^{5}.$
We note that by conditions $a_{12}a_{1}-a_{11}a_{4}\geq0$ and
$a_{10}a_{1}-a_{10}a_{4}-a_{12}a_{1}+a_{11}a_{4}\geq0$ we have
$a_{4}\leq a_{1}.$

5) If $x_{1}\neq0$. In the case $x_{2}=0$ and $x_{3}\neq0$,
let $x_{6}=\lambda\neq0$. From $x_{3}^{(1)}=x_{3}$ we get $x_{1}=\frac{a_{8}+a_{9}-a_{7}\lambda}{a_{3}}$,
from $x_{1}^{(1)}=x_{1}$ we deduce $x_{3}=\frac{a_{1}\lambda-a_{4}}{a_{3}}$,
from $x_{4}^{(1)}=x_{4}$ we get $x_{5}=\frac{a_{12}-a_{11}\lambda}{a_{10}}$
and from $x_{5}^{(1)}=x_{5}$ we obtain $x_{4}=\frac{a_{4}x_{1}+a_{9}x_{3}}{a_{10}x_{5}}=\frac{a_{4}a_{8}+\lambda(a_{1}a_{9}-a_{4}a_{7})}{a_{3}(a_{12}-a_{11}\lambda)}$.
By replacing these values in $x_{1}+\cdots+x_{6}=1$ we obtain
a quadratic equation (\ref{qe}) and by solving it we find $\lambda$
and for the positiveness of each coordinate get boundary conditions
for $\lambda$ and thus, we obtain $\lambda_{5}.$

6) Let $x_{1}\neq0$, $x_{2}\neq0$. If $x_{3}=0$, with $x_{6}=\lambda\neq0$,
from $x_{2}^{(1)}=x_{2}$ we deduce $x_{1}=\frac{a_{5}+a_{6}}{a_{2}}$,
from $x_{1}^{(1)}=x_{1}$ we get $x_{2}=\frac{a_{1}\lambda-a_{4}}{a_{2}}$,
from $x_{4}^{(1)}=x_{4}$ it comes $x_{5}=\frac{a_{12}-a_{11}\lambda}{a_{10}}$
and from $x_{5}^{(1)}=x_{5}$ it follows $x_{4}=\frac{a_{4}x_{1}+a_{5}x_{2}}{a_{10}x_{5}}
=\frac{a_{4}a_{6}+a_{1}a_{5}\lambda}{a_{2}(a_{12}-a_{11}\lambda)}$.
Here also by $x_{1}+\cdots+x_{6}=1$ we deduce that $\lambda$
is root of a quadratic equation (\ref{qea}), by solving
it we find $\lambda$ and for the positiveness of each coordinate
we get boundary conditions for $\lambda$ and finally, we get
$\lambda_{6}.$

7) Here more difficult case is when all coordinates of fixed point
are non-zero. Let $x_{1}\neq0,x_{2}\neq0,x_{3}=\lambda\neq0,x_{4}\neq0,x_{5}\neq0,x_{6}\neq0.$
Then from $x_{2}^{(1)}=x_{2}$ we get $x_{1}=\frac{a_{5}+a_{6}}{a_{2}}$,
from $x_{3}^{(1)}=x_{3}$ it follows $x_{6}=\frac{a_{2}(a_{8}+a_{9})-a_{3}(a_{5}+a_{6})}{a_{2}a_{7}}$,
from $x_{1}^{(1)}=x_{1}$ we deduce $x_{2}=\frac{a_{1}x_{6}-a_{4}-a_{3}\lambda}{a_{2}}$,
and from $x_{4}^{(1)}=x_{4}$ and $x_{5}^{(1)}=x_{5}$ we deduce
respectively $x_{5}=\frac{a_{12}-a_{11}x_{6}}{a_{10}}$ and $x_{4}=\frac{a_{4}x_{1}+a_{5}x_{2}+a_{9}\lambda}{a_{12}-a_{11}x_{6}}$
. By $x_{1}+\cdots+x_{6}=1$ we have linear equation with respect
to $\lambda$ and for positiveness of coordinates we get conditions
to the parameters. We note that if we request positiveness, then
by $x_{1}+\cdots+x_{6}=1$ we do not need to show coordinates
less than 1. Thus, we found all possible fixed points of the
operator (\ref{disc}).
\end{proof}
\begin{rem}
We note that all coordinates of each fixed point must belong
to $[0,1]$ and from the vertices of the simplex only $(0,0,0,0,1,0)$
and $(0,0,0,0,0,1)$ can be fixed points (see $\lambda_{1}$).
\end{rem}

Next we study the type of fixed points. First, we find the Jacobian
of the operator (\ref{disc}):

\[
\left[\begin{array}{cccccc}
\tfrac{\partial x_{1}^{(1)}}{\partial x_{1}} & -a_{2}x_{1} & -a_{3}x_{1} & 0 & 0 & a_{1}x_{1}\\
a_{2}x_{2} & \tfrac{\partial x_{2}^{(1)}}{\partial x_{2}} & 0 & 0 & 0 & 0\\
a_{3}x_{3} & 0 & \tfrac{\partial x_{3}^{(1)}}{\partial x_{3}} & 0 & 0 & a_{7}x_{3}\\
0 & 0 & 0 & \tfrac{\partial x_{4}^{(1)}}{\partial x_{4}} & a_{10}x_{4} & a_{11}x_{4}\\
a_{4} & a_{5} & a_{9} & -a_{10}x_{5} & 1-a_{10}x_{4} & 0\\
-a_{1}x_{6} & a_{6} & -a_{7}x_{6}+a_{8} & -a_{11}x_{6}+a_{12} & 0 & \tfrac{\partial x_{6}^{(1)}}{\partial x_{6}}
\end{array}\right]
\]
where
\[
\begin{array}{cc}
\begin{aligned}\tfrac{\partial x_{1}^{(1)}}{\partial x_{1}} & =1-a_{4}+a_{1}x_{6}-a_{2}x_{2}-a_{3}x_{3},\\
\tfrac{\partial x_{3}^{(1)}}{\partial x_{3}} & =1-a_{8}-a_{9}+a_{3}x_{1}+a_{7}x_{6},\\
\tfrac{\partial x_{6}^{(1)}}{\partial x_{6}} & =1-a_{1}x_{1}-a_{7}x_{3}-a_{11}x_{4}.
\end{aligned}
 & \begin{aligned}\tfrac{\partial x_{2}^{(1)}}{\partial x_{2}} & =1-a_{5}-a_{6}+a_{2}x_{1},\\
\tfrac{\partial x_{4}^{(1)}}{\partial x_{4}} & =1-a_{12}+a_{10}x_{5}+a_{11}x_{6},\\
\\
\end{aligned}
\end{array}
\]

By solving simple determinant of 6th order, we can define that
at the fixed point $\lambda_{1}$ the Jacobian has two eigenvalues
equal to 1, so this fixed point is non-hyperbolic.

At the fixed point $\lambda_{2}$ the Jacobian is:
\begin{align*}
\det\left(J(\lambda_{2})-\mu I\right)= & \left(1-\mu\right)\left(1-a_{4}+\tfrac{a_{1}a_{12}}{a_{11}}-\mu\right)\left(1-a_{8}-a_{9}+\tfrac{a_{7}a_{12}}{a_{11}}-\mu\right)\times\\
 & \left(1-a_{5}-a_{6}-\mu\right)\left(1-a_{10}+\tfrac{a_{10}a_{12}}{a_{11}}-\mu\right)\left(1-a_{11}+a_{12}-\mu\right),
\end{align*}
from this, one eigenvalue $\mu_{4}=1$, so the fixed point $\lambda_{2}$
also non-hyperbolic. In the general case, finding the eigenvalues
of Jacobian for the other fixed points is more difficult.\medskip{}

Let $e_{i}=(\delta_{1i},...,\delta_{mi})\in S^{m-1},i=1,...,m$ be
the vertices of the simplex $S^{m-1},$ where $\delta_{ij}$
the Kronecker's symbol, and by $\mathcal{V}_{l}$ we denote the set of
all $l$-Volterra QSOs.
\begin{prop}
\label{pRZ}(\cite{RZ})

\par (1) The vertex $e_{i}$ is a fixed point for an $\ell$-Volterra
QSO if and only if $P_{ii,i}=1,(i=1,...,m).$

\par (2) For any collection $I_{s}=\{e_{i_{1}},...,e_{i_{s}}\}\subset\{e_{l+1},...,e_{m}\},(s\leq m-1)$
there exists a family $\mathcal{V}_{l}(I_{s})\subset\mathcal{V}_{l}$
such that $\{e_{i_{1}},...,e_{i_{s}}\}$ is an $s$-cycle for
each $V\in\mathcal{V}_{l}(I_{s}).$
\end{prop}

Using this proposition and the conditions to the parameters (\ref{cond})
we formulate the following proposition.
\begin{prop}
\textcompwordmark{}

\par (1) The following vertices $e_{1},e_{2},e_{3}$ and $e_{4}$
can not be fixed points for the QSO (\ref{disc})

\par (2) Vertices $\{e_5, e_6\}$, and any collection $\{e_{i_{1}},...,e_{i_{s}}\}\subset\{e_{1},...,e_{4}\}$, $(1<s\leq 4)$ can not be
a periodic orbit for the QSO (\ref{disc}).
\end{prop}

\begin{proof}
(1) By Proposition \ref{pRZ} the vertices $e_{1},e_{2},e_{3}$
and $e_{4}$ are fixed points if and only if $P_{11,1}=1,P_{22,2}=1,P_{33,3}=1$
and $P_{44,4}=1$ respectively. But for the operator (\ref{disc})
all parameters are positive and $P_{ii,i}\neq1$ for any $i=1,2,3,4,$
so $e_{i}$ can not be fixed point for any $i=1,2,3,4.$\\
 (2) By Proposition \ref{pRZ} only the collection $\{e_{5},e_{6}\}$
can be 2-cycle for the 4-Volterra QSO. But for the 4-Volterra
operator (\ref{disc}), $V(e_{5})=e_{5}\neq e_{6}$ and $V(e_{6})=e_{6}\neq e_{5}.$
The remaining part of this assertion follows from Remark 3.4 in \cite{RZ}.
The proposition is proved.
\end{proof}

\section{The limit points of trajectories}

In this section we step by step study the limit behavior of trajectories of initial points $\lambda^{(0)}\in S^5$ under
 operator (\ref{disc}), i.e the sequence $V^n(\lambda^{(0)})$, $n\geq 1$. Note that since $V$ is a continuous operator,
  its trajectories have
as a limit some fixed points obtained in Proposition \ref{fp}.

\subsection{Case no dissolved inorganic matter.}

We study here the case where in the model there is no dissolved
inorganic matter.
\begin{prop}
For an initial point $\lambda^{(0)}=\left(x_{1}^{(0)},x_{2}^{(0)},x_{3}^{(0)},x_{4}^{(0)},x_{5}^{(0)},x_{6}^{(0)}\right)\in S^{5}$
(except fixed points), with $x_{6}^{\left(0\right)}=x_{6}^{\left(1\right)}=0$, the trajectory (under action of operator (\ref{disc}))
has the following limit
\[
\lim_{n\to\infty}{V^{(n)}}(\lambda^{(0)})=\left(0,0,0,0,1,0\right).
\]
\end{prop}

\begin{proof}
If $x_{6}^{\left(0\right)}=x_{6}^{\left(1\right)}=0$, from the
sixth equation of (\ref{disc}) we deduce $x_{2}^{\left(0\right)}=x_{3}^{\left(0\right)}=x_{4}^{\left(0\right)}=0$
which implies according to the second, third and fourth equations
of (\ref{disc}) $x_{2}^{\left(n\right)}=x_{3}^{\left(n\right)}=x_{4}^{\left(n\right)}=0$
for every $n\geq0$. With these results the sixth equation of
(\ref{disc}) becomes $x_{6}^{\left(n+1\right)}=x_{6}^{\left(n\right)}\left(1-a_{1}x_{1}^{\left(n\right)}\right)$
from $x_{6}^{\left(1\right)}=0$ we deduce recursively that $x_{6}^{\left(n\right)}=0$
for every $n\geq0$. Then the first equation of (\ref{disc})
becomes $x_{1}^{\left(n+1\right)}=\left(1-a_{4}\right)x_{1}^{\left(n\right)}$
which implies $x_{1}^{\left(n\right)}=\left(1-a_{4}\right)^{n}x_{1}^{\left(0\right)}$
with $1-a_{4}<1$ thus $\lim_{n\to\infty}x_{1}^{(n)}=0$ and
therefore $\lim_{n\to\infty}x_{5}^{(n)}=1$.
\end{proof}
{\it Biological interpretation}: Without zooplankton, mixoplankton and bacteria there is no more
production of dissolved inorganic matter. Due to absence of food
the amount of phytoplankton decreases until it disappears, as
the dissolved organic matter is no longer degraded by bacteria,
it accumulates in the environment until it is saturated.

\subsection{Case no dissolved organic matter. }
\begin{prop} For an initial point $\lambda^{(0)}=\left(x_{1}^{(0)},x_{2}^{(0)},x_{3}^{(0)},x_{4}^{(0)},x_{5}^{(0)},x_{6}^{(0)}\right)\in S^{5}$  (except fixed points),
with $x_{5}^{\left(0\right)}=x_{5}^{\left(1\right)}=0$,  the trajectory
has the following limit
\[
\lim_{n\to\infty}V^{(n)}(\lambda^{(0)})=\begin{cases}
\left(0,0,0,0,0,1\right) & \text{if }a_{11}\leq a_{12}\\
\bigl(0,0,0,1-\frac{a_{12}}{a_{11}},0,\frac{a_{12}}{a_{11}}\bigr) & \text{if }a_{11}>a_{12}.
\end{cases}
\]
\end{prop}

\begin{proof}
From $x_{5}^{\left(0\right)}=x_{5}^{\left(1\right)}=0$ and the
fifth equation of (\ref{disc}) we get $x_{1}^{\left(0\right)}=x_{2}^{\left(0\right)}=x_{3}^{\left(0\right)}=0$
which implies recursively according to the first three equations
of (\ref{disc}) that $x_{1}^{\left(n\right)}=x_{2}^{\left(n\right)}=x_{3}^{\left(n\right)}=0$
for any $n\geq0$, consequently from the fifth equation we get
$x_{5}^{\left(n\right)}=0$ for any $n\geq0$. Using these results
the system (\ref{disc}) is reduced to
\[
\left\{ \begin{alignedat}{1}x_{4}^{\left(n+1\right)} & =x_{4}^{\left(n\right)}\left(1-a_{12}+a_{11}x_{6}^{\left(n\right)}\right)\\
x_{6}^{\left(n+1\right)} & =x_{6}^{\left(n\right)}+x_{4}^{\left(n\right)}\left(a_{12}-a_{11}x_{6}^{\left(n\right)}\right).
\end{alignedat}
\right.
\]

We see that $x_{4}^{\left(n+1\right)}+x_{6}^{\left(n+1\right)}=x_{4}^{\left(n\right)}+x_{6}^{\left(n\right)}$
thus $x_{6}^{\left(n\right)}=1-x_{4}^{\left(n\right)}$ for any
$n\geq1$, with this we get $x_{4}^{\left(n+1\right)}=x_{4}^{\left(n\right)}\left(1+a_{11}-a_{12}-a_{11}x_{4}^{\left(n\right)}\right)$.
Consider the function $f(x)=x(1+a_{11}-a_{12}-a_{11}x)$, it
has two fixed points $0$ and $1-\frac{a_{12}}{a_{11}}$. Since
$f'(0)=1-a_{12}+a_{11}$ the fixed point $0$ is attracting if
and only if $a_{11}<a_{12}$, non-hyperbolic if $a_{11}=a_{12}$
and repelling if $a_{11}>a_{12}$. Similarly, by $f'(1-\frac{a_{12}}{a_{11}})=1+a_{12}-a_{11}$
we have that $1-\frac{a_{12}}{a_{11}}$ is attracting if $a_{11}>a_{12}$
and non-hyperbolic if $a_{11}=a_{12}$. Moreover, for $a_{11}\leq a_{12}$
the sequence $x_{4}^{(n)}=f^{n}(x^{(0)})$ is decreasing. Finally
for any $x_{4}^{(0)}\in[0,1]$ we have
\[
\lim_{n\to\infty}f^{(n)}\left(x_{4}^{\left(0\right)}\right)=\begin{cases}
0, & \text{if }a_{11}\leq a_{12}\\
1-\frac{a_{12}}{a_{11}}, & \text{if }a_{11}>a_{12}.
\end{cases}
\]
\end{proof}
{\it Biological interpretation}: The absence of dissolved organic matter is due to the absence
of plankton. When the bacteria disappearance rate is higher than
their consumption rate of dissolved inorganic matter, bacteria
disappear completely. When the comparison of these rates is reversed,
the composition of the environment tends towards a balance between
bacteria and dissolved inorganic matter.

\subsection{Case no bacteria.}
\begin{prop}
For any initial point $\lambda^{(0)}=\left(x_{1}^{(0)},x_{2}^{(0)},x_{3}^{(0)}, 0, x_{5}^{(0)},x_{6}^{(0)}\right)\in S^{5}$  (except fixed points), the trajectory has the following limit
\[
\lim_{n\to\infty}V^{(n)}(\lambda^{(0)})=\left(0,0,0,0,\overline{\lambda},1-\overline{\lambda}\right),
\]

where $\overline{\lambda}=\overline{\lambda}(\lambda^{0})$.
\end{prop}

\begin{proof}
When $x_{4}^{\left(0\right)}=0$, according to the fourth equation
of (\ref{disc}) we have $x_{4}^{\left(n\right)}=0$ for all
$n\geq0$. Then $x_{5}^{(n+1)}=x_{5}^{\left(n\right)}+a_{4}x_{1}^{\left(n\right)}+a_{5}x_{2}^{\left(n\right)}+a_{9}x_{3}^{\left(n\right)}\geq x_{5}^{\left(n\right)}$
so, the sequence $x_{5}^{(n)}$ has a limit $\overline{\lambda}$ and $\lim_{n\to\infty}\bigl(a_{4}x_{1}^{\left(n\right)}+a_{5}x_{2}^{\left(n\right)}+a_{9}x_{3}^{\left(n\right)}\bigr)=0$,
but $a_{4}x_{1}^{\left(n\right)}+a_{5}x_{2}^{\left(n\right)}+a_{9}x_{3}^{\left(n\right)}\geq a_{4}x_{1}^{\left(n\right)}\geq0$
which leads to $\lim_{n\to\infty}x_{1}^{\left(n\right)}=0$ and
in a similar way $\lim_{n\to\infty}x_{2}^{\left(n\right)}=\lim_{n\to\infty}x_{3}^{\left(n\right)}=0$.
Then with $x_{6}^{\left(n\right)}=1-\sum_{i=1}^{5}x_{i}^{\left(n\right)}$
we deduce that $\lim_{n\to\infty}x_{6}^{\left(n\right)}=1-\overline{\lambda}$.
\end{proof}
{\it Biological interpretation}: In the absence of bacteria, compartments
between dissolved inorganic and organic matters are not connected,
the trophic network is open. Dissolved organic matter (DOM) resulting
from plankton mortality are no longer are no longer transformed
into dissolved inorganic matter (DIM), the environment is no
longer supplied with DIM. The supply of phytoplankton and mixoplankton
decreases the amount of DIM which causes an increase in plankton
mortality and contributes to increase the amount of DOM until
reaching a point of equilibrium between the amounts of DOM and
DIM.

\subsection{Case no phytoplankton. }
The following lemma is useful.
\begin{lem}
\label{3-lim}  For an initial point $\lambda^{(0)}=\left(x_{1}^{(0)},x_{2}^{(0)},x_{3}^{(0)},x_{4}^{(0)},x_{5}^{(0)},x_{6}^{(0)}\right)\in S^{5}$  (except fixed points), if $x_{4}^{\left(n\right)}$ converges to zero and two among the sequences $x_{1}^{\left(n\right)}$,
$x_{2}^{\left(n\right)}$, $x_{3}^{\left(n\right)}$
have zero limit then remaining (third)  sequence has zero limit and the limit
of the trajectory is the fixed point $\lambda_{1}$.
\end{lem}

\begin{proof}
-- If $\lim_{n\to\infty}x_{i}^{(n)}=0$ for $i=2,3$.

Suppose that $\lim_{n\to\infty}x_{1}^{(n)}=\nu \neq 0$. Then from
the first equation of (\ref{disc}) we get $\lim_{n\to\infty}x_{6}^{(n)}=\frac{a_{4}}{a_{1}}$,
which by the sixth equation of (\ref{disc})
gives that $a_{1}\nu = 0$, thus we have a contradiction.

\medskip{}

-- If $\lim_{n\to\infty}x_{i}^{(n)}=0$ for $i=1,3.$

Suppose that $\lim_{n\to\infty}x_{2}^{(n)}\neq 0$, then from
the second equation of (\ref{disc}) we deduce $\lim_{n\to\infty}x_{1}^{(n)}=\frac{a_{5}+a_{6}}{a_{2}}\neq 0$, this is
a contradiction.

-- If $\lim_{n\to\infty}x_{i}^{(n)}=0$ for $i=1,2.$

Suppose that $\lim_{n\to\infty}x_{3}^{(n)}=\mu\neq 0$, then it
follows from the third equation of (\ref{disc}) that $\lim_{n\to\infty}x_{6}^{(n)}=\frac{a_{8}+a_{9}}{a_{7}}$.
Now by the equality
$$x_{5}^{\left(n\right)}=1-\sum_{{i=1\atop i\ne 5}}^{6}x_{i}^{\left(n\right)}$$
we deduce that the sequence $x_{5}^{\left(n\right)}$ has a limit. Consequently,
from the fifth equation of (\ref{disc}) we
get $a_{9}\mu=0$, again a contradiction.
\end{proof}
{\it Biological interpretation}: If bacteria and  two species  of plankton disappear then the third plankton also
disappears.

\begin{lem}
\label{lemma2}  For an initial point $\lambda^{(0)}=\left(x_{1}^{(0)},x_{2}^{(0)},x_{3}^{(0)},x_{4}^{(0)},x_{5}^{(0)},x_{6}^{(0)}\right)\in S^{5}$  (except fixed points), if $a_{10}+a_{11}\leq a_{12}$ (resp. $a_{3}+a_{7}\leq a_{8}+a_9$) then $x_{4}^{\left(n\right)}$  (resp.$x_{3}^{\left(n\right)}$) converges to zero.
\end{lem}
\begin{proof}
From the condition $a_{10}+a_{11}\leq a_{12}$ we have  $a_{12}-a_{10}x_{5}^{(0)}-a_{11}x_{6}^{(0)}>0,$ because, $x_5=x_6=1$ never occurs, and by Proposition \ref{pc} it follows $a_{12}\leq1,$ from this we get  $a_{12}-a_{10}x_{5}^{(0)}-a_{11}x_{6}^{(0)}\leq1,$ it means that $0\leq 1-(a_{12}-a_{10}x_{5}^{(0)}-a_{11}x_{6}^{(0)})<1,$ so we have

$$\lim_{n\to\infty}x_{4}^{(n)}=0.$$
Similarly, if $a_{3}+a_{7}\leq a_{8}+a_9$ then $\lim_{n\to\infty}x_{3}^{(n)}=0.$
\end{proof}

\begin{prop}
\label{k1} If  $a_{7}\leq a_{8}+a_{9}$
and $a_{10}+a_{11}\leq a_{12}$ then for any initial point $\lambda^{(0)}=\left(0,x_{2}^{0},x_{3}^{0},x_{4}^{0},x_{5}^{0},x_{6}^{0}\right)\in S^{5}$
(except fixed points), the trajectory has the following
limit
\[
\lim_{n\to\infty}V^{(n)}(\lambda^{(0)})=\bigl(0,0,0,0,\overline{\lambda},1-\overline{\lambda}\bigr).
\]
where $\overline{\lambda}=\overline{\lambda}(\lambda^{(0)}).$
\end{prop}

\begin{proof}
If $x_{1}^{\left(0\right)}=0$ according to (\ref{disc}) we
have $x_{1}^{\left(n\right)}=0$ for all $n\geq0$. Then the
operator $V$ becomes
\begin{equation}
V:\left\{ \begin{alignedat}{1}x'_{2} & =x_{2}\left(1-a_{5}-a_{6}\right)\\
x'_{3} & =x_{3}\left(1-a_{8}-a_{9}+a_{7}x_{6}\right)\\
x'_{4} & =x_{4}\left(1-a_{12}+a_{10}x_{5}+a_{11}x_{6}\right)\\
x'_{5} & =x_{5}+a_{5}x_{2}+a_{9}x_{3}-a_{10}x_{4}x_{5}\\
x'_{6} & =x_{6}\left(1-a_{7}x_{3}-a_{11}x_{4}\right)+a_{6}x_{2}+a_{8}x_{3}+a_{12}x_{4}
\end{alignedat}
\right.\label{c1}
\end{equation}

We have $x_{2}^{\left(n+1\right)}=x_{2}^{\left(0\right)}\left(1-a_{5}-a_{6}\right)^{n}$
since $1-a_{5}-a_{6}<1$ the sequence $x_{2}^{(n)}$ has limit
$\lim_{n\to\infty}x_{2}^{(n)}=0$. Moreover, from condition $a_{10}+a_{11}\leq a_{12}$ and according Lemma \ref{lemma2} we have $\lim_{n\to\infty}x_{4}^{(n)}=0.$   \medskip{}

 If $a_{7}\leq a_{8}+a_{9}$ then  $x_{3}^{\left(1\right)}\leq x_{3}^{\left(0\right)}\left(1-(a_{8}+a_{9}-a_7x_6)\right)\leq x_{3}^{(0)}.$ So the sequence $x_{3}^{\left(n\right)}$ has limit and according Lemma \ref{3-lim} it has zero limit.

From the fourth equation of the system (\ref{c1}) we deduce
\[
\left|x_{5}^{\left(n+1\right)}-x_{5}^{\left(n\right)}\right|\leq a_{5}x_{2}^{\left(n\right)}+a_{9}x_{3}^{\left(n\right)}+a_{10}x_{4}^{\left(n\right)}.
\]

Then for all $m>n$ we have
\[
\left|x_{5}^{\left(m\right)}-x_{5}^{\left(n\right)}\right|\leq\sum_{k=n}^{m-1}\left|x_{5}^{\left(k+1\right)}-x_{5}^{\left(k\right)}\right|\leq a_{5}\sum_{k=n}^{m}x_{2}^{\left(k\right)}+a_{9}\sum_{k=n}^{m}x_{3}^{\left(k\right)}+a_{10}\sum_{k=n}^{m}x_{4}^{\left(k\right)}
\]
using the upper bounds of sequences $x_{2}^{\left(n\right)}$,
$x_{3}^{\left(n\right)}$, $x_{4}^{\left(n\right)}$ obtained
previously, we get
\begin{align*}
\left|x_{5}^{\left(m\right)}-x_{5}^{\left(n\right)}\right| & \leq a_{5}x_{2}^{\left(0\right)}\sum_{k=n}^{+\infty}\left(1-a_{5}-a_{6}\right)^{k}+a_{9}x_{3}^{\left(0\right)}\sum_{k=n}^{+\infty}\left(1+a_{7}-a_{8}-a_{9}\right)^{k}+\\
 & \qquad a_{10}x_{4}^{\left(0\right)}\sum_{k=n}^{+\infty}\left(1+a_{10}+a_{11}-a_{12}\right)^{n}
\end{align*}
from this we deduce that $x_{5}^{\left(n\right)}$ is a Cauchy
sequence, thus we can say that this sequence has limit $\overline{\lambda}$,
but it always depends the initial point $\lambda^{(0)}.$ In this case
we know that operator called \emph{regular}, i.e., every initial
point has its limit point.

\medskip{}

\end{proof}
{\it Biological interpretation}: If no phytoplankton
and if death rates of mixoplankton and bacteria are higher than
their consumption rates of dissolved inorganic compounds, then
after some time the populations of zooplankton, mixoplankton
and bacteria gradually disappears and finally it stay only dissolved
inorganic and organic matters, of course they depend the initial
amount of each other organisms.

\subsection{Case neither zooplankton nor mixoplankton. }
\begin{prop}
If  $a_{1}\leq a_{4}$
and $a_{10}+a_{11}\leq a_{12}$ then for any initial point $\lambda^{(0)}=\left(x_{1}^{(0)},0,0,x_{4}^{(0)},x_{5}^{(0)},x_{6}^{(0)}\right)\in S^{5}$
(except fixed points), the trajectory has the following
limit
\[
\lim_{n\to\infty}V^{(n)}(\lambda^{(0)})=\left(0,0,0,0,\overline{\lambda},1-\overline{\lambda}\right),
\]
where $\overline{\lambda}=\overline{\lambda}(\lambda^{(0)}).$
\end{prop}

\begin{proof}
If $x_{2}^{\left(0\right)}=x_{3}^{\left(0\right)}=0$ then $x_{2}^{\left(n\right)}=x_{3}^{\left(n\right)}=0$
for all $n\geq0$, the system (\ref{disc}) becomes

\begin{equation}
\left\{ \begin{alignedat}{1}x'_{1} & =x_{1}\left(1-a_{4}+a_{1}x_{6}\right)\\
x'_{4} & =x_{4}\left(1-a_{12}+a_{10}x_{5}+a_{11}x_{6}\right)\\
x'_{5} & =x_{5}+a_{4}x_{1}-a_{10}x_{4}x_{5}\\
x'_{6} & =x_{6}\left(1-a_{1}x_{1}-a_{11}x_{4}\right)+a_{12}x_{4}
\end{alignedat}
\right.\label{c1-1}
\end{equation}

By the condition $a_{10}+a_{11}\leq a_{12}$ using Lemma \ref{lemma2} we have $\lim_{n\to\infty}x_{4}^{(n)}=0.$
If  $a_{1}\leq a_{4}$  then $x_{1}^{\left(1\right)}\leq\left(1-(a_{4}-a_{1}x_6)\right)x_{1}^{\left(0\right)}\leq x_{1}^{\left(0\right)}$ so the sequence  $x_{1}^{\left(n\right)}$ has limit. Thus, by Lemma \ref{3-lim} we have $\lim_{n\to\infty}x_{1}^{(n)}=0.$
From the fourth equation of (\ref{c1-1}) we get
\[
\left|x_{6}^{\left(n+1\right)}-x_{6}^{\left(n\right)}\right|\leq a_{12}x_{4}^{\left(n\right)}\leq a_{12}\left(1+a_{10}+a_{11}-a_{12}\right)^{n}x_{4}^{\left(0\right)},
\]
therefore for any $m>n$ we have
\[
\left|x_{6}^{\left(m\right)}-x_{6}^{\left(n\right)}\right|\leq a_{12}\sum_{k=n}^{+\infty}a_{12}\left(1+a_{10}+a_{11}-a_{12}\right)^{k}x_{4}^{\left(0\right)},
\]
so $x_{6}^{\left(n\right)}$ is a Cauchy sequence thus the sequence
$x_{6}^{\left(n\right)}$ has a limit.\medskip{}

\end{proof}
{\it Biological interpretation}: Without zooplankton and mixoplankton, if the consumption rates
of dissolved inorganic compounds of phytoplankton and bacteria
is less than their death rates, then after some time the populations
of phytoplankton and bacteria gradually disappears, ultimately
it only stays in the environment dissolved inorganic and organic
matters whose proportions whose proportions depend of course
on the initial conditions.

\subsection{Cases with three types of plankton and bacteria at initial stage.}

We formulate a proposition for general case where all coordinates of an initial point are non zero.
\begin{prop}
If $a_{1}\leq a_{4}$, $a_{2}\leq a_{5}+a_{6}$, $a_{3}+a_{7}\leq a_{8}+a_{9}$,
$a_{10}+a_{11}\leq a_{12}$ then  for any
initial point $\lambda^{(0)}=(x_{1}^{(0)},x_{2}^{(0)},x_{3}^{(0)},x_{4}^{(0)},x_{5}^{(0)},x_{6}^{(0)})\in S^{5}$ (except fixed points)
 the following holds
\[
\lim_{n\to\infty}{V^{(n)}}(\lambda^{(0)})=(0,0,0,0,\overline{\lambda},1-\overline{\lambda}),
\]
where $\overline{\lambda}=\overline{\lambda}(\lambda^{(0)}).$
\end{prop}

\begin{proof}
From conditions $a_{3}+a_{7}\leq a_{8}+a_{9}$, $a_{10}+a_{11}\leq a_{12}$ according Lemma \ref{lemma2} we obtain $x_{3}^{(n)}$ and $x_{4}^{(n)}$ have zero limits. We will study the 4 cases according to whether the inequalities $a_{1}\leq a_{4}$, $a_{2}\leq a_{5}+a_{6}$  are equalities or strict inequalities.\medskip{}

If $a_{1}\leq a_{4}$, $a_{2}\leq a_{5}+a_{6}$,  we deduce from the first two equations
of (\ref{disc}) the following inequalities
\begin{align}
\begin{aligned}x_{1}^{(1)} & \leq\left(1+a_{1}-a_{4}\right)x_{1},\\
\end{aligned}
\quad & \begin{aligned}x_{2}^{(1)} & \leq\left(1+a_{2}-a_{5}-a_{6}\right)x_{2},\\
\end{aligned}
\label{ineq-1}
\end{align}

If $a_{1}=a_{4}$, from (\ref{ineq-1}) the sequence $x_{1}^{\left(n\right)}$
is decreasing so it has a limit noted $\mu_{1}$. Analogously,
if $a_{2}=a_{5}+a_{6}$ the sequence $x_{2}^{\left(n\right)}$
has a limit $\mu_{2}$.\medskip{}

\emph{Case 1}: Two strict inequalities.

If $a_{1}<a_{4}$, $a_{2}<a_{5}+a_{6}$ from the first two equations of (\ref{disc})
we get
\begin{align}
\begin{aligned}x_{1}^{(n)} & \leq\left(1+a_{1}-a_{4}\right)^{n}x_{1}^{(0)},\\
\end{aligned}
\quad & \begin{aligned}x_{2}^{(n)} & \leq\left(1+a_{2}-a_{5}-a_{6}\right)^{n}x_{2}^{(0)}.\\
\end{aligned}
\label{ineq-2}
\end{align}
 So the sequences $x_{1}^{(n)},x_{2}^{(n)}$ converge to zero, and the limit of the trajectory is the fixed
point of type $\lambda_{1}$ (cf. Proposition \ref{fp}).\medskip{}

\emph{Case 2}: One equality and one strict inequality.

In each of the two possible cases, one
of the sequences $x_{1}^{\left(n\right)}$, $x_{2}^{\left(n\right)}$
 has limit zero, therefore according to Lemma \ref{3-lim}, the remaining
sequence also has limit zero, thus the limit of the trajectory is the fixed point of type $\lambda_{1}$.\medskip{}

\emph{Case 3}: Two equalities.

3.1. If $a_{1}=a_{4}$ and $a_{2}=a_{5}+a_{6}$.

We suppose $\mu_{1}\neq0$ and $\mu_{2}\neq0$, otherwise according
to Lemma \ref{3-lim} we have $\lim_{n\to\infty}x_{i}^{(n)}=0$
for $1\leq i\leq4$ and the limit of the operator $V$ is the
fixed point of type $\lambda_{1}$. Then from the second equation
of (\ref{disc}) we get $\mu_{1}=\frac{a_{5}+a_{6}}{a_{2}}=1$
and thus $\mu_{2}=0$, this is contradiction.
Consequently,  the limit of the trajectory is the fixed point of type $\lambda_{1}$.\medskip{}

\end{proof}

\section{Conjecture for more general cases.}

Recall that $\lambda_{2}, \lambda_{3}, \lambda_{4}$ are the fixed points defined in Proposition \ref{fixp}.

Based on the numerical analysis, the explicit form of
the operator and the fixed points we make the following conjectures:

{\bf Conjecture 1.}
Let the initial point $\lambda^{(0)}=(x_{1}^{(0)}, x_{2}^{(0)}, x_{3}^{(0)}, x_{4}^{(0)}, x_{5}^{(0)}, x_{6}^{(0)})$, with $x_{4}^{0}>0$ and
  $a_{11}>a_{12}$,  $a_{3}\leq a_{7}\leq a_{8}+a_{9}$ (or $x_{3}^{(0)}=0$).
Then

(1) If $x_{1}^{(0)}>0,\ \ a_{4}<a_{1},$ and $a_{12}a_{1}-a_{11}a_{4}>0,\ \ a_{10}a_{1}-a_{10}a_{4}-a_{12}a_{1}+a_{11}a_{4}>0$
then

\[
\lim_{n\to\infty}{V^{(n)}}(\lambda^{(0)})=\lambda_{4}.
\]

(2) If $a_{12}a_{1}-a_{11}a_{4}\leq0$ then
\[
\lim_{n\to\infty}{V^{(n)}}(\lambda^{(0)})=\lambda_{2}.
\]

{\bf Conjecture 2.}
If $a_{11}>a_{12}$, $a_{7}>a_{8}+a_{9}$, $a_{1}\leq a_{4}$ (or $x_{1}^{(0)}=0$),  $a_{12}a_{7}-a_{11}(a_{8}+a_{9})>0$, $a_{10}(a_{7}-(a_{8}+a_{9}))-(a_{12}a_{7}-a_{11}(a_{8}+a_{9}))>0$,
 $x_{3}^{(0)}>0$, $x_{4}^{(0)}>0$,
then
\[
\lim_{n\to\infty}{V^{(n)}}(\lambda^{(0)})=\lambda_{3}.
\]

We note that for the other cases an investigation of the limit behavior
of the operator (\ref{disc}) seems very difficult.

\section*{Acknowledgements}

Shoyimardonov thanks the "El-Yurt Umidi" Foundation under the
Cabinet of Ministers of the Republic of Uzbekistan for financial
support during his visit to the Institut Montpelliérain Alexander
Grothendieck in University of Montpellier (France) and prof.
R. Varro for the invitation.

\end{document}